\documentclass[12pt]{article}
\usepackage{geometry}[margin=1in]
\usepackage[utf8]{inputenc}
\usepackage{amsfonts}
\usepackage{xcolor,comment}
\usepackage{amsmath}
\usepackage{authblk}
\usepackage{amsthm}
\usepackage{appendix}
\newtheorem{theorem}{Theorem}[section]
\newtheorem{definition}[theorem]{Definition}
\newtheorem{remark}[theorem]{Remark}
\newtheorem{proposition}[theorem]{Proposition}
\newtheorem{corollary}[theorem]{Corollary}
\newtheorem{lemma}[theorem]{Lemma}

\newtheorem*{motivating-example}{A Motivating Example}
\title{The Small-Noise Limit of the Most Likely Element is the Most Likely Element in the Small-Noise Limit}
\author[1]{Zachary Selk}
\author[2]{Harsha Honnappa}
\affil[1]{Department of Mathematics, Queen's University zachary.selk@queensu.ca}
\affil[2]{Department of Industrial Engineering, Purdue University}

\date{}

\begin{document}

\maketitle
\begin{abstract}

In this paper, we study the Onsager-Machlup function and its relationship to the Freidlin-Wentzell function for measures equivalent to arbitrary infinite dimensional Gaussian measures. The Onsager-Machlup function can serve as a density on infinite dimensional spaces, where a uniform measure does not exist, and has been seen as the Lagrangian for the ``most likely element". The Freidlin-Wentzell rate function is the large deviations rate function for small-noise limits and has also been identified as a Lagrangian for the ``most likely element". This leads to a conundrum - what is the relationship between these two functions?

We show both pointwise and $\Gamma$-convergence (which is essentially the convergence of minimizers) of the Onsager-Machlup function under the small-noise limit to the Freidlin-Wentzell function - and give an expression for both. That is, we show that the small-noise limit of the most likely element is the most likely element in the small noise limit for infinite dimensional measures that are equivalent to a Gaussian. Examples of measures include the law of solutions to path-dependent stochastic differential equations and the law of an infinite system of random algebraic equations.


\end{abstract}

\section{Introduction}

The primary objective of this paper is an investigation of the relationship between the Onsager-Machlup \cite{Onsager-Machlup-Original} and Freidlin-Wentzell~\cite{Freidlin-Wentzell} functions for measures equivalent to an arbitrary Gaussian measure. 

First, recall the setting of $\mathbb R^d$-valued diffusion processes. By Girsanov's theorem \cite{Oksendal}, the law of the solution to the stochastic differential equation (SDE)
\[dX^\varepsilon=b(X^\varepsilon) dt+\varepsilon dB(t)\]
is equivalent to the law of the Gaussian process $\varepsilon B(t)$, in the sense that both measures agree on the same null sets. In `small noise' settings as $\varepsilon\to 0$, there has been extensive work deriving the so-called `Friedlin-Wentzell' large deviations rate functions for diffusions (see \cite{LDP-Dembo-Zeitouni}, as well as \cite{Guo-1,Budhiraja-LDP-FBM} for more recent work), and it can be shown that the Friedlin-Wentzell function for $X^\varepsilon$ is given by $$\operatorname{FW}_X(z)=\frac{1}{2}\int_0^T(z'(t)-b(z(t))^2dt.$$ A well-known interpretation of the minimizer of this Friedlin-Wentzell function is as the ``most likely" path the small noise process takes between fixed initial and final states. On the other hand, the minimizer of the Onsager-Machlup function~\cite{Durr-Bach}  for $X^\varepsilon$ ($\varepsilon > 0$), $$\operatorname{OM}_{X^\varepsilon}(z)=\frac{1}{2\varepsilon^2}\int_0^T [(z'(t)-b(z(t))^2+\varepsilon^2 b'(z(t))]dt,$$
when $z$ is sufficiently regular, is also interpreted as the most likely path between initial and final states. This immediately poses a dilemma - if both theories claim to produce the most likely path, how can they be reconciled? 

In the SDE setting, the relationship between these two theories has been explored in \cite{Li-and-Li,Stuart-Gamma}. As can be immediately anticipated from the displays above, $\varepsilon^2 \operatorname{OM}_{X^\varepsilon}$ converges pointwise to $\operatorname{FW}_X$ as $\varepsilon \to 0$. In these papers the authors further prove that 
the Onsager-Machlup function $\Gamma$-converges (\cite{Gamma-book}) to the Freidlin-Wentzell function. Thus, in particular for $X^\varepsilon$,
%for a family of stochastic differential equations
%\[dX^\varepsilon=b(X^\varepsilon) dt+\varepsilon dB(t),\]
 their results imply that $\varepsilon^2\operatorname{OM}_{X^\varepsilon}(z)$ $\Gamma$-converges to the $\operatorname{FW}_X(z)$ as $\varepsilon \to 0$. Since $\Gamma$-convergence can be understood as essentially the convergence of minimizers, it follows, at least in the case of SDEs, that the Freidlin-Wentzell ``most likely path" is the small noise limit of the Onsager-Machlup ``most likely path". 

However, the connection between Onsager-Machlup and Freidlin-Wentzell theories appears to be more subtle. For instance, in the paper \cite{Dutra} the authors study numerics for estimating the most likely path of a given stochastic differential equation. The authors show that the choice of discretization leads to differing behavior - an Euler-Maruyama discretization led to the Freidlin-Wentzell ``most likely path", while a trapezoid discretization led to the Onsager-Machlup ``most likely path". Further, connections between the Onsager-Machlup and Friedlin-Wentzell theories have been hinted at in more general settings. For instance, in the context of computing rare event paths for stochastic partial differential equations (SPDEs),~\cite{FW-WRV-E} observe that the Friedlin-Wentzell path is an ``analogue'' of the Onsager-Machlup path (see~\cite{LDP-SPDE} as well). However, to the authors' best knowledge, no work has been done to rigorously establish the relationship between the Onsager-Machlup and Freidlin-Wentzell functions (and their minimizers) at this level of generality.

By a sufficiently general version of Girsanov's theorem, in many settings the law of a solution to SDEs/SPDEs driven by a general Gaussian noise (which can be a stochastic process, random field or more general objects) is equivalent to the law of the underlying noise. Utilizing this, our primary result below encompasses the setting of SDEs/SPDEs, and extends well beyond to that of measures equivalent to arbitrary Gaussian reference measures.%  We establish a relationship between the Onsager-Machlup and Freidlin-Wentzell functions at the level of generality of Gaussian measures, stated below.

\begin{theorem}\label{theorem:main}
    Let $\mathcal B$ be a separable Banach space with centered Gaussian measure $\mu_0$. Let $\mu$ be another Borel measure on $\mathcal B$ equivalent to $\mu_0$, with density $\frac{d\mu}{d\mu_0}=\exp(\Phi)$, where $\Phi$ satisfies mild regularity conditions. Define the measures $\mu_0^\varepsilon$ by $\mu_0^\varepsilon(A)=\mu_0(\frac{1}\varepsilon A)$ for Borel $A\subset \mathcal B$ and $\mu^\varepsilon = \exp(\frac{1}{\varepsilon^2} \Phi) \mu_0^\varepsilon.$ Then the Onsager-Machlup function for the measures $\mu^\varepsilon$ exist, denoted by $\operatorname{OM}_{\mu^\varepsilon}$. Additionally, we have that $\{\mu^\varepsilon\}$ satisfy a LDP with rate function $\operatorname{FW}(z):=\lim_{\varepsilon\to 0^+} \varepsilon^2 \operatorname{OM}_{\mu^\varepsilon}(z)$ and speed $\varepsilon^2$. 
    
    Furthermore, denoting by $\operatorname{OM-Mode}(\mu^\varepsilon):=\arg \inf_{z\in \mathcal B} \operatorname{OM}_{\mu^\varepsilon}$ we have that every cluster point of the elements $\operatorname{OM-Mode}(\mu^\varepsilon)$ is a minimizer of $\operatorname{FW}$, denoted by $\operatorname{FW-Mode}(\mu):=\arg\inf_{z\in \mathcal B} \operatorname{FW}(z)$. If 
$\varepsilon^2\operatorname{OM}_{\mu^\varepsilon}$ are equicoercive, then we also have that
    $\lim_{\varepsilon\to 0^+}\operatorname{OM-Mode}(\mu^\varepsilon)= \operatorname{FW-Mode}(\mu).$
\end{theorem}

The proof of Theorem \ref{theorem:main} boils down to two ``tilting" lemmas - one for Onsager-Machlup in Corollary~\ref{corollary:OM-Gaussian-Equivalent} and one for large deviations in Lemma~\ref{lemma:tilt-general}. In large deviations analysis, the exponential tilting principle is a mechanism for transferring large deviations principles from a given sequence of measures to a sequence of equivalent measures. Lemma~\ref{lemma:tilt-general} proves a slightly generalized version of this principle that applies in the setting of measures equivalent to a Gaussian. In Corollary~\ref{corollary:OM-Gaussian-Equivalent}, we present the Onsager-Machlup function for measures equivalent to Gaussians, and represents a type of tilting. Our proof uses these two results to establish the main Theorem \ref{theorem:main}.
\subsection{Literature Review}
We provide a short review of the literature related to the Onsager-Machlup and Friedlin-Wentzell theories; a full review is outside the scope of this short paper. The Onsager-Machlup theory was originally introduced in \cite{Onsager-Machlup-Original} to compute the probability that a system experiencing Gaussian (thermodynamic) fluctuations will pass through a succession of non-equilibrium states over time. This probability was expressed in terms of what is now recognized as the Friedlin-Wentzell function. The Onsager-Machlup theory for SDEs was originally studied in \cite{Durr-Bach}, who generalized the results in~\cite{Onsager-Machlup-Original} to non-Gaussian Markov diffusions. Durr and Bach provide a rigorous definition and derivation of the Onsager-Machlup function . Subsequently, the theory has been substantially developed and extended to random fields and SPDEs~\cite{dz-rf,wz-spde} as well. A key motivation for the continued development of the Onsager-Machlup theory is its central role in  metastability analysis \cite{meta-1,meta-2,meta-3,meta-4,meta-5}, and the fact that the Onsager-Machlup function can be viewed as the Lagrangian yielding the most likely path connecting metastable states in a stochastic system.  

The Friedlin-Wentzell function arises as the small noise large deviations rate function of a Markov diffusion process~\cite{LDP-Dembo-Zeitouni,Freidlin-Wentzell} . The Friedlin-Wentzell theory is also used in pathwise analyses of metastable behavior, with the large deviations rate function begin used to determine the time-scale of `tunneling' behavior in the small noise setting~\cite{ov-ldm}. Metastable phenomena emerge outside the small noise setting, and therefore the relationship between the Onsager-Machlup and Friedlin-Wentzell theories is of significance and interest. The convergence of the Onsager-Machlup to Freidlin-Wentzell for SDEs, specifically, was studied in \cite{Li-and-Li} and \cite{Du-Li-Li-Ren}. This paper establishes the relationship between them for measures that are equivalent to Gaussian measures on Banach spaces.

\section{The Onsager-Machlup Formalism}
In finite dimensional probability, computations involving expectations, probabilities and related quantities can be greatly eased by working with a probability density function. 
These densities capture how a probability distribution compares with some kind of uniform measure - typically Lebesgue or counting measure. Densities are also the optimization objective to be maximized when computing the mode (or most likely element) of a distribution.

For instance, given a probability measure $\mu$ on $\mathbb R^n$ that is equivalent to the Lebesgue measure $\lambda$, we can express its Radon-Nikodym derivative by
$$\frac{d\mu}{d\lambda}(z)=\lim_{\varepsilon\to 0}\frac{\mu(B_\varepsilon(z))}{\lambda(B_\varepsilon(z))}=\lim_{\varepsilon\to 0}\frac{\mu(B_\varepsilon(z))}{\lambda(B_\varepsilon(0))}\propto \lim_{\varepsilon\to 0}\frac{\mu(B_\varepsilon(z))}{\mu(B_\varepsilon(0))},$$
where the final relation comes from the Lebesgue differentiation theorem. The last expression makes sense even on non-locally compact spaces where there might not be a uniform measure. 

The notion of a density does not immediately transfer to infinite dimensional probability as there is no uniform measure on infinite dimensional spaces. However, the so-called Onsager-Machlup function can serve the role of a density in infinite dimensions. This function was introduced by Onsager and Machlup in \cite{Onsager-Machlup-Original}, and the key insight in Onsager-Machlup theory is that one can compare a probability distribution to translations of itself rather than comparing a probability distribution to some translation invariant measure. This recovers the standard density in the finite dimensional case but allows for ``densities'' on infinite dimensional spaces.

\begin{definition}\label{def:OM}
Let $(X,d)$ be a metric space. Let $\mu$ be a Borel probability measure on $X$. If the following limit exists
\begin{equation}
    \lim_{\varepsilon\to 0}\frac{\mu(B_\varepsilon(z_1))}{\mu(B_\varepsilon(z_2))}=\exp\left(\operatorname{OM}_\mu(z_2)-\operatorname{OM}_\mu(z_1)\right),
\end{equation}
then $\operatorname{OM}_\mu(z)$ is called the Onsager-Machlup function for $\mu$.
\end{definition}

\begin{remark}
The Onsager-Machlup function in Definition \ref{def:OM} can be thought of as the negative log ``density" of $\mu$. That is, 
\[``\frac{d\mu}{d\lambda}(z)"=\exp\left(-\operatorname{OM}_\mu(z)\right)\]
for some possibly nonexistent uniform measure $\lambda$. In the case where $X=\mathbb R^d$ then the above equality is rigorous, by the Lebesgue differentiation theorem. Also, note that the Onsager-Machlup function is only defined up to an additive constant. We also define the ``mode" or the most likely element of $\mu$ as the minimizer of $\operatorname{OM}_\mu$.
\end{remark}

\begin{remark}
%As noted before, one important use of the Onsager-Machlup function is to serve as the objective to optimize when computing the mode or the ``most likely" element of a probability measure on an infinite dimensional space. This is of particular interest when looking for the ``most likely path" of a stochastic process. This was the approach originally taken in \cite{Durr-Bach}, where the authors computed the Onsager-Machlup function for solutions to stochastic differential equations. %They demonstrated that the optimizers of this function served as the ``most likely path" of this diffusion. 

 We also observe that the Onsager-Machlup formalism has found its way into Bayesian statistics and MAP estimation such as in \cite{Stuart2,Stuart1}. Additionally, in \cite{Self1} the authors prove a ``portmanteau" theorem that relates the Onsager-Machlup function on an abstract Banach space equipped with a Gaussian measure to an information projection problem, to an ``open loop" or state-independent KL-weighted control problem, and in the case of classical Wiener space to an Euler-Lagrange equation or variational form. Furthermore, using this Portmanteau theorem the authors in \cite{self2} prove a Feynman-Kac type result for systems of ordinary differential equations. They demonstrate that the solution to a system of second order and linear ordinary differential equations is the most likely path of a diffusion. This Feynman-Kac result, like the original Feynman-Kac for parabolic partial differential equations, can (in principle) be used to efficiently solve systems of ordinary differential equations via Monte Carlo methods.
\end{remark}

The following proposition represents a ``tilting'' lemma, allowing for Onsager-Machlup functions to be transfered to equivalent measures.

\begin{lemma}\label{prop:tilt-for-OM}
Let $\mu_0$ be a Borel measure on Banach space $\mathcal B$. Suppose that $\mu_0$ has an associated Onsager-Machlup function $\operatorname{OM}_{\mu_0}:\mathcal B\to [-\infty,\infty]$. Consider the measure $\mu$ with density
\begin{equation*}
    \frac{d\mu}{d\mu_0}=\frac{1}{E_{\mu_0}[e^{-\Phi}]}\exp(-\Phi).
\end{equation*}
Suppose that for each $\varepsilon_0>0$, for each $x\in \mathcal B$ and for all $\varepsilon<\varepsilon_0$ there is some continuous increasing $\phi_{x,\varepsilon_0}:[0,\varepsilon_0]\to [0,\infty)$ with $\phi(0)=0$ so that $|\Phi(u)-\Phi(x)|\leq \phi(\varepsilon)$ on $B_\varepsilon(x)$. Then $\mu$ has an associated Onsager-Machlup function 
\begin{equation*}
    \operatorname{OM}_\Phi(z)=\Phi(z)+\operatorname{OM}_{\mu_0}(z). 
\end{equation*}
\end{lemma}
\begin{proof}
We consider the ratio
\begin{equation*}
    \frac{\mu(B_{\varepsilon}(z_1))}{\mu(B_{\varepsilon}(z_2))}=\frac{\int_{B_\varepsilon(z_1)}\mu(du)}{\int_{B_\varepsilon(z_2)}\mu(du)}.
\end{equation*}
Using the density, adding and subtracting $\Phi(z_i)$ for $i=1,2$ in both integrals yields that
\begin{align*}
     \frac{\mu(B_{\varepsilon}(z_1))}{\mu(B_{\varepsilon}(z_2))}&=\frac{\int_{B_\varepsilon(z_1)}\exp(-\Phi(u))\mu_0(du)}{\int_{B_\varepsilon(z_2)}\exp(-\Phi(u))\mu_0(du)}\\
     &=\frac{\int_{B_\varepsilon(z_1)}\exp(-\Phi(u)+\Phi(z_1)-\Phi(z_1))\mu_0(du)}{\int_{B_\varepsilon(z_2)}\exp(-\Phi(u)+\Phi(z_2)-\Phi(z_2))\mu_0(du)}\\
     &=\exp\left(\Phi(z_2)-\Phi(z_1)\right)\frac{\int_{B_\varepsilon(z_1)}\exp(-\Phi(u)+\Phi(z_1))\mu_0(du)}{\int_{B_\varepsilon(z_2)}\exp(-\Phi(u)+\Phi(z_2))\mu_0(du)}.
\end{align*}
By assumption, there are some $\phi_i$ on $B_\varepsilon(z_i)$ so that
 \begin{align*}
    |\Phi(z_i)-\Phi(u)|\leq \phi_i( \varepsilon)
\end{align*}
for $i=1,2$. Therefore for $\phi=\phi_1-\phi_2$ we have that
\begin{equation*}
    \frac{\mu_0(B_{\varepsilon}(z_1))}{\mu_0(B_{\varepsilon}(z_2))} e^{-\phi(\varepsilon)} \leq \frac{\int_{B_\varepsilon(z_1)}\exp(-\Phi(u)+\Phi(z_1))\mu_0(du)}{\int_{B_\varepsilon(z_2)}\exp(-\Phi(u)+\Phi(z_2))\mu_0(du)} \leq  \frac{\mu_0(B_{\varepsilon}(z_1))}{\mu_0(B_{\varepsilon}(z_2))} e^{\phi(\varepsilon)}.
\end{equation*}
Taking the limit $\varepsilon\to 0$ concludes. 
\end{proof}

For the purposes of our paper, we are interested in Gaussian measures. For more information on Gaussian measure theory see \cite{Bogachev, hairer2009introduction}. For measures equivalent to a Gaussian~\cite[Theorem 3.2]{Stuart-OM} derives an expression for the Onsager-Machlup function, which we recall in the next proposition. 
\begin{proposition}~\label{prop:OM-Gaussian-Equivalent}
Let $\mu_0$ be a centered Gaussian measure on Banach space $\mathcal B$ with Cameron-Martin space $\mathcal H_{\mu_0}$ and Cameron-Martin norm $\|\cdot\|_{\mu_0}$. Let $\Phi:\mathcal B\to \mathbb R$ be a function that is locally Lipschitz and locally bounded. Define the measure $\mu$ with positive density $\frac{d\mu}{d\mu_0}=\frac{1}{E_{\mu_0}[e^{-\Phi}]}e^{-\Phi}$. Then the Onsager-Machlup function for $\mu$ exists and is equal to
\begin{equation*}
\operatorname{OM}_{\mu}(z) = 
    \begin{cases}
    \Phi(z)+\frac{1}{2}\|z\|_{\mu_0}^2&\text{ if } z\in \mathcal H_\mu\\
    \infty&\text{ else }.
    \end{cases}
\end{equation*}
\end{proposition}
\begin{proof}
    We recall (see e.g. \cite{Bogachev} Section 4.7) that the Onsager-Machlup function for the Gaussian measure $\mu_0$ is 
    \[
        \operatorname{OM}_{\mu_0}(z) = \begin{cases}
            \frac{1}{2} \|z\|_{\mu_0}^2 &~if~z \in \mathcal H_{\mu_0}\\
            \infty &~{else}.
        \end{cases}
    \]
    The expression for $\operatorname{OM}_\mu$ follows by Lemma~\ref{prop:tilt-for-OM}. 
\end{proof}

The following corollary can be straightforwardly proved by substitution in Proposition~\ref{prop:OM-Gaussian-Equivalent}, and contains the assumptions we need for Theorem \ref{theorem:main}.

\begin{corollary}[$\varepsilon$-dependent Tilting Lemma]~\label{corollary:OM-Gaussian-Equivalent}
Let $\mu_0$ be a centered Gaussian measure on Banach space $\mathcal B$ with Cameron-Martin space $\mathcal H_{\mu_0}$ and Cameron-Martin norm $\|\cdot\|_{\mu_0}$. Consider the functions $F^\varepsilon(y):\mathcal B\to \mathbb R$ and suppose that they satisfy the following expansion
\begin{equation*}
    F^\varepsilon(y)=F_0(y)+\varepsilon F_1(y)+\frac{\varepsilon^2}{2}F_2(y)+...+\varepsilon^n R_n(\varepsilon, y),
\end{equation*}
for some functions $F_i:\mathcal B\to \mathbb R$ with $\lim_{\varepsilon\to 0} R_n(\varepsilon, y)=0$. Suppose that $F_0$ is continuous. Furthermore, assume that the functions $F_i$ and $R_n$ satisfy the following moment condition 
\begin{equation*}
    \limsup_{\varepsilon\to 0}\varepsilon^2 \log E_{\mu_0^\varepsilon}\left[\exp\left(\gamma_i \frac{\max\{|F_i(y)|,|R_n(\varepsilon,y)|\}}{\varepsilon^2}\right)\right]<\infty,
\end{equation*}
for some $\gamma_i>0$ and for all $0 \leq i\leq n$. Define the measures equivalent to $\mu_0^\varepsilon$ by 
\begin{equation*}
    \mu^\varepsilon=\frac{1}{E_{\mu_0^\varepsilon}\left[\exp\left(-\frac{1}{\varepsilon^2}F^\varepsilon(y)\right)\right]}\exp\left(-\frac{1}{\varepsilon^2}F^\varepsilon(y)\right)\mu_0^\varepsilon.
\end{equation*}
Then the Onsager-Machlup function for $\mu^\varepsilon$ exists and is equal to
\begin{equation*}
\operatorname{OM}_{\mu^\varepsilon}(z) = 
    \begin{cases}
   \frac{1}{\varepsilon^2} F^\varepsilon(z)+\frac{1}{2\varepsilon^2}\|z\|_{\mu_0}^2&\text{ if } z\in \mathcal H_{\mu_0}\\
    \infty&\text{ else }.
    \end{cases}
\end{equation*}

\end{corollary}

\begin{comment}
    
\begin{remark}
    Note that even though the Definition \ref{def:OM} for Onsager-Machlup depends on the norm on the space, the expressions we have given are independent of the norm. The only place that the norm affects anything is on the regularity properties of $f_z$ or of $\Phi$. 
\end{remark}

As the Onsager-Machlup function plays the role of the density in infinite dimensions, one might ask to what extent it determines the measure. As it turns out, in each equivalence class of measures the Onsager-Machlup function determines the measure as the below proposition will show.

\begin{proposition}
     Let $\mu_1,\mu_2$ be two Borel probability measures on some Banach space $\mathcal{B}$. Suppose that $\operatorname{OM}_{\mu_1}(z)=\operatorname{OM}_{\mu_2}(z)$ for all equivalent shifts $z\in H\subset \mathcal B$ for some dense $H$ common to both $\mu_1$ and $\mu_2$. Assume that $\mu_1$ and $\mu_2$ are equivalent with the log of their Radon-Nikodym derivative satisfying the assumptions in Proposition \ref{prop:tilt-for-OM}, then $\mu_1=\mu_2$.
\end{proposition}
\begin{proof} 
By Proposition \ref{prop:tilt-for-OM}, we have that 
$\operatorname{OM}_{\mu_1}(z)=\operatorname{OM}_{\mu_2}(z)=-\log\left(\frac{d\mu_1}{d\mu_2}\right)+\operatorname{OM}_{\mu_1}(z)$ for all $z\in H$. $\frac{d\mu_1}{d\mu_2}$ is continuous and $H$ is dense in $\mathcal B$ so $\frac{d\mu_1}{d\mu_2}=1$ identically. 
\end{proof}
\end{comment}

\section{An $\varepsilon$-dependent Tilting Lemma and Small Noise LDPs}

There are multiple ways of constructing new LDPs from existing ones. One principled approach is ``tilting" - which passes large deviations principles from a sequence of reference measures to a sequence of measures equivalent to the reference measures. We direct the reader to \cite[Theorem III.17]{den-hollander} , for the standard tilting lemma. One form of the standard tilting lemma reads as follows.

\begin{lemma}
    Let $\mu_0^\varepsilon$ be a collection of Borel probability measures on a Banach space $\mathcal B$ satisfying a LDP with good rate function $I_0:\mathcal B\to [0,\infty]$ and rate $\varepsilon^2$. Consider a continuous function $F:\mathcal B\to \mathbb R$. Assume that for some $\gamma>0$ we have
    \begin{equation*}
        \limsup_{\varepsilon\to 0}\varepsilon \log E_{\mu_0^\varepsilon}\left[\exp\left(\gamma |F(y)|\right)\right]<\infty.
    \end{equation*}
    Define the measures equivalent to $\mu_0^\varepsilon$ by 
    \begin{equation*}
        \mu^\varepsilon=\frac{1}{E_{\mu_0^\varepsilon}[\exp\left(-\frac{1}{\varepsilon^2}F(y)\right)]}\exp\left(-\frac{1}{\varepsilon^2}F(y)\right)\mu_0^\varepsilon.
    \end{equation*}
    Then $\mu^\varepsilon$ satisfies a LDP with good rate function 
    \begin{equation*}
    I(y):=I_0(y)+F(y) -\inf_{z\in \mathcal B}\{F(z)+I_0(z)\}.
    \end{equation*}
\end{lemma}

However, in many cases where one would like to apply the tilting lemma, the function $F$ depends on $\varepsilon$. This is apparent in the case of Freidlin-Wentzell large deviations for stochastic differential equations as we will see shortly. Therefore, we need a generalized version of the tilting lemma which we provide. 
\begin{lemma}[$\varepsilon$-dependent Tilting Lemma]\label{lemma:tilt-general}
Let $\mu_0^\varepsilon$ be a collection of exponentially tight Borel measures on Banach space $\mathcal B$ satisfying a LDP with good rate function $I_0:\mathcal B\to [0,\infty]$. 
Consider the functions $F^\varepsilon(y):\mathcal B\to \mathbb R$ and suppose that they satisfy the following expansion
\begin{equation*}
    F^\varepsilon(y)=F_0(y)+\varepsilon F_1(y)+\frac{\varepsilon^2}{2}F_2(y)+...+\varepsilon^n R_n(\varepsilon, y),
\end{equation*}
for some functions $F_i:\mathcal B\to \mathbb R$ with $\lim_{\varepsilon\to 0} R_n(\varepsilon, y)=0$. Suppose that $F_0$ is continuous. Furthermore, assume that the functions $F_i$ and $R_n$ satisfy the following moment condition 
\begin{equation*}
    \limsup_{\varepsilon\to 0}\varepsilon^2 \log E_{\mu_0^\varepsilon}\left[\exp\left(\gamma_i \frac{\max\{|F_i(y)|,|R_n(\varepsilon,y)|\}}{\varepsilon^2}\right)\right]<\infty,
\end{equation*}
for some $\gamma_i>0$ and for all $0 \leq i\leq n$. Define the measures equivalent to $\mu_0^\varepsilon$ by 
\begin{equation*}
    \mu^\varepsilon=\frac{1}{E_{\mu_0^\varepsilon}\left[\exp\left(-\frac{1}{\varepsilon^2}F^\varepsilon(y)\right)\right]}\exp\left(-\frac{1}{\varepsilon^2}F^\varepsilon(y)\right)\mu_0^\varepsilon.
\end{equation*}
Assume that the measures $\mu^\varepsilon$ are exponentially tight. Then $\mu^\varepsilon$ satisfies a LDP with good rate function \begin{equation*}
    I(y):=I_0(y)+F_0(y) -\inf_{z\in \mathcal B}\{F_0(z)+I_0(z)\}.
\end{equation*}
\end{lemma}
\begin{proof}
We will apply Bryc's lemma (see \cite[Theorem 4.4.2]{LDP-Dembo-Zeitouni}). To this end, consider a bounded and continuous function $\varphi:\mathcal B\to \mathbb R$. Then consider $$L:=\varepsilon^2 \log \left(E_{\mu^\varepsilon}\left[\exp\left(-\frac{\varphi(y)}{\varepsilon^2}\right)\right]\right).$$
Using the form of $F^\varepsilon$, we get that
\begin{align*}
    L&= \varepsilon^2 \log \left(E_{\mu_0^\varepsilon}\left[\exp\left(-\frac{\varphi(y)+F^\varepsilon(y)}{\varepsilon^2}\right)\right]\right)-\varepsilon^2 \log \left(E_{\mu_0^\varepsilon}\left[\exp\left(-\frac{F^\varepsilon(y)}{\varepsilon^2}\right)\right]\right).
    \end{align*}
In the limit $\varepsilon\to 0$, by H\"older's inequality, reverse H\"older's inequality, Varadhan's lemma and the assumptions on $F_i$ and $R_n$, we have that
\begin{equation}\label{eq:proof-LDP-1}
   \lim_{\varepsilon\to 0} L = \lim_{\varepsilon\to 0} \varepsilon^2 \log \left(E_{\mu_0^\varepsilon}\left[\exp\left(-\frac{\varphi(y)+F_0(y)}{\varepsilon^2}\right)\right]\right)-\varepsilon^2 \log \left(E_{\mu_0^\varepsilon}\left[\exp\left(-\frac{F_0(y)}{\varepsilon^2}\right)\right]\right).
\end{equation}By Varadhan's lemma \cite[Theorem 4.3.1]{LDP-Dembo-Zeitouni} and the assumptions on $F_0$, we have that 
\begin{equation*}
    \lim_{\varepsilon\to 0} L=-\inf_{y\in \mathcal B}\{\varphi(y)+F_0(y)+I_0(y)\}+\inf_{z\in \mathcal B}\{F_0(z)+I_0(z)\}.
\end{equation*}
By Bryc's lemma \cite[Theorem 4.4.2]{LDP-Dembo-Zeitouni} and the assumption of tightness of $\mu^\varepsilon$, we conclude.
\end{proof}

Small noise large deviations for arbitrary Gaussian measures on Banach space are well known. In \cite{Bogachev}, for instance, it is shown that the Freidlin-Wentzell rate function for a general Gaussian measure $\mu_0$ with Cameron-Martin space $\mathcal H_{\mu_0}$ is $\frac{1}2\|\cdot\|_{\mu_0}^2$. In particular, the small noise rate function for $\varepsilon B(t)$ is 
\begin{equation*}
    \operatorname{FW}(z)=
    \begin{cases}
    \frac{1}2\int_0^T (z'(t))^2 dt&\text{ for }z\in \mathcal W_0^{1,2}\\
    \infty &\text{ else },
    \end{cases}
\end{equation*}
where $\mathcal W_0^{1,2}=\{t\mapsto \int_0^t f(s) ds: \int_0^T f^2(s) ds<\infty\}$ is the Cameron-Martin space of the classical Wiener measure which consists of absolutely continuous functions with $L^2$ weak derivative. The first instance of small noise large deviations for infinite-dimensional non-Gaussian measures came in \cite{Freidlin-Wentzell}, where the authors studied small noise large deviations for the solution to the SDE
\begin{equation}~\label{eq:sn-sde}
    dX^\varepsilon(t)=b(X^\varepsilon(t))dt+\varepsilon dB(t),
\end{equation}
where $b \in C^1$. In particular \cite{Freidlin-Wentzell} shows the following proposition.

\begin{proposition}
    The law of $X^\varepsilon$ satisfies a LDP as $\varepsilon \to 0$, with speed $\varepsilon^2$ and with rate function 
\begin{equation*}
    \operatorname{FW}(z)=
    \begin{cases}
    \frac{1}2\int_0^T (b(z(t))-z'(t))^2 dt&\text{ for }z\in \mathcal W_0^{1,2}\\
    \infty &\text{ else }.
    \end{cases}
\end{equation*}
\end{proposition}

Below, we offer an alternate, considerably simpler, proof of this result by invoking the $\varepsilon$-tilting lemma above. 
\begin{proof}
By Girsanov, the law of $X^\varepsilon$, $\mu^\varepsilon$, has density with respect to the law of $\varepsilon B(t)$, $\mu_0^\varepsilon$ given by
\begin{equation}\label{eq:density-for-sde}
    \frac{d\mu^\varepsilon}{d\mu_0^\varepsilon}=\exp\left(\frac{1}{\varepsilon^2}\left(\int_0^T b(B(t)) dB(t)-\frac{1}2\int_0^T b^2(B(t))dt\right)\right).
\end{equation}
At first glance, it might appear from equation \eqref{eq:density-for-sde} that we might not need the full $\varepsilon$-dependent Lemma \ref{lemma:tilt-general}. However, note that the It\^o integral is not defined pathwise. On the other hand, recall that the It\^o integral is $\mu_0^\varepsilon$-a.s. equal to a Stratonovich integral, which is defined pathwise. Applying It\^o's lemma under $\mu_0^\varepsilon$ yields
\begin{equation*}
    \frac{d\mu^\varepsilon}{d\mu_0^\varepsilon}=\exp\left(\frac{1}{\varepsilon^2}\left(\int_0^T b(B(t)) \circ  dB(t)-\frac{\varepsilon^2}2\int_0^T b'(B(t))dt-\frac{1}2\int_0^T b^2(B(t))dt\right)\right),
\end{equation*}
where $\int_0^T b(B(t)) \circ  dB(t)$ represents the Stratonovich integral. As $b$ is continuous it has an antiderivative $F$, and the Stratonovich integral satisfies $\int_0^T b(B(t)) \circ  dB(t)=F(B(T))-F(B(0))$, which is indeed continuous. 

Applying the $\varepsilon$-dependent tilting Lemma \ref{lemma:tilt-general} gives that the rate function for $\mu^\varepsilon$ is
\begin{equation*}
    \operatorname{FW}(z)=
    \begin{cases}
    -\left(\int_0^T b(z(t)) dz(t)-\frac{1}2\int_0^T b^2(z(t))dt\right)+\frac12 \int_0^T (z'(t))^2dt&\text{ if } z\in \mathcal W_0^{1,2}\\
    \infty &\text{ else }.
    \end{cases}
\end{equation*}
\end{proof}

\section{Proof of Theorem~\ref{theorem:main}}
We begin this section with a motivating example for small noise large deviations for Gaussian measures in finite dimensions.

\begin{motivating-example}~\label{ref:motivating-example}
Consider a family of real valued normally distributed random variables $X^\varepsilon$ where $X^\varepsilon\sim \mathcal N(0,\varepsilon^2)$. We are interested in the decay of the probability 
\begin{equation*}
    P( X^\varepsilon\in A)=\int_A \frac{1}{\sqrt{2\pi \varepsilon^2}}e^{-x^2/2\varepsilon^2} dx,
\end{equation*}
for Borel $A\subset \mathbb R$. Thankfully the standard Laplace principle on $\mathbb R$ yields the appropriate scaling and gives the large deviations principle 
\begin{equation*}
    \lim_{\varepsilon\to 0^+}\varepsilon^2 \log P( X^\varepsilon\in A)=-\operatorname{essinf}_{x\in A} \frac{x^2}{2}.
\end{equation*}
In this case, the Onsager-Machlup function for the law of the random variable $X^\varepsilon$ is just the term in the exponent - $\operatorname{OM}_{\mu^\varepsilon}(x)=\frac{x^2}{2\varepsilon^2}$ and we have that $\varepsilon^2 \operatorname{OM}_{\mu^\varepsilon}(x)=\frac{x^2}{2}=\operatorname{FW}(x).$
\end{motivating-example}

Perhaps surprisingly, this equivalence (up to an $\varepsilon^2$ scaling) of  the Onsager-Machlup and Freidlin-Wentzell functions holds true even for a Gaussian measure on a Banach space. More precisely, consider a Banach space $(\mathcal B, \|\cdot\|)$ with a Borel measure $\mu$ so that all the one dimensional projections are Gaussian. Consider the measure $\mu^\varepsilon$ defined by $\mu^\varepsilon(A):=\mu(\varepsilon^{-1}A)$ for Borel $A\subset \mathcal B$. Then it is shown in e.g. \cite{Bogachev} section 4.9, that the measures $\mu^\varepsilon$ satisfy a large deviations principle (LDP) with rate function $\operatorname{FW}_\mu$ and additionally for regular enough $z$ we have that, for any $\varepsilon > 0$,

\[\varepsilon^2\operatorname{OM}_{\mu^\varepsilon}(z)=\operatorname{FW}_\mu(z)=\frac{1}{2}\|z\|_\mu^2,\]
where $\|\cdot\|_\mu$ is the so-called Cameron-Martin norm (distinct from the norm on the Banach space). 

 To obtain an analogous result on arbitrary Banach spaces, first recall the definition of $\Gamma$-convergence (for more information on $\Gamma$-convergence see \cite{Gamma-book}):

\begin{definition}
    Let $X$ be a topological space and $\mathcal{N}(x)$ denote the set of all neighborhoods of $x\in X$. Further, let $F_n : X \to \overline{\mathbb R}$ be a sequence of functions on $X$. The $\Gamma$-lower and $\Gamma$-upper limits are defined as
    \begin{align*}
        \Gamma-\liminf_{n\to\infty} F_n(x) &= \sup_{N_x \in \mathcal N(x)} \liminf_{n\to\infty} \inf_{y\in N_x} F_n(y),\\
        \Gamma-\limsup_{n\to\infty} F_n(x) &= \sup_{N_x \in \mathcal N(x)} \limsup_{n\to\infty} \inf_{y\in N_x} F_n(y).
    \end{align*}
    Then, the function $F : X \to \overline{\mathbb R}$ is a $\Gamma$-limit of $F_n$ if $\Gamma-\liminf_{n\to\infty} F_n = \Gamma-\limsup_{n\to\infty} F_n = F$.
\end{definition}

Recall the small-noise SDE in~\eqref{eq:sn-sde}, whose measure $\mu^\varepsilon$ satisfies~\eqref{eq:density-for-sde}. Applying Proposition \ref{corollary:OM-Gaussian-Equivalent} shows that the Onsager-Machlup function for $\mu^\varepsilon$ is
\begin{align*}
   \operatorname{OM}_{\mu^\varepsilon}(z)&= -\frac{1}{\varepsilon^2}\left(\int_0^T b(z(t)) dz(t)-\frac{\varepsilon^2}2\int_0^T b'(z(t))dt-\frac{1}2\int_0^T b^2(z(t))dt\right)\\
   &~+\frac{1}{2\varepsilon^2} \int_0^T (z'(t))^2dt,
\end{align*}
for $z\in \mathcal W_0^{1,2}$ and infinite otherwise. It is not hard to see that $\varepsilon^2 \operatorname{OM}_{\mu^\varepsilon}$ converges to $\operatorname{FW}$ both as a pointwise and as a $\Gamma$-limit.

The following proposition shows that this conclusion can be generalized considerably. 
\begin{proposition}\label{prop:FW-first}
Let $\mu_0$ be a centered Gaussian measure on a Banach space $\mathcal B$ with Cameron-Martin norm $\|\cdot \|_{\mu_0}$ (see \cite{hairer2009introduction}, Section 3.2).  Let $\mu_0^\varepsilon$ denote the measure defined by $\mu_0^\varepsilon(A)=\mu_0(\varepsilon^{-1} A)$ for Borel $A\subset \mathcal B$. Let $F:\mathcal B \to \mathbb R$ be a function and suppose that $F(y)=F_0(y)+\varepsilon F_1(y)+...+\varepsilon^n R_n(\varepsilon,y):=F^\varepsilon(y)$, $\mu_0^\varepsilon$-a.s. where the $F_i$ and $R_n$ satisfy the assumptions of Lemma \ref{lemma:tilt-general}, and suppose that the $F_i$ are locally bounded. Define the measures $\mu^\varepsilon$ with densities
\begin{equation*}
    \frac{d\mu^\varepsilon}{d\mu_0^\varepsilon}=\frac{1}{E_{\mu_0^\varepsilon}[e^{-\frac{1}{\varepsilon^2}F^\varepsilon(y)}]}\exp\left(-\frac{1}{\varepsilon^2}F^\varepsilon(y)\right).
\end{equation*}

Then the Freidlin-Wentzell rate function for $\mu^\varepsilon$ exists and is
\begin{equation*}
    \operatorname{FW}(y)=F_0(y)+\frac{1}{2}\|y\|_{\mu_0}^2-\inf_{z\in \mathcal H_{\mu_0}}(F_0(z)+\frac{1}{2}\|z\|_{\mu_0}^2).
\end{equation*}
Furthermore, denote by $\operatorname{OM}_\varepsilon(z)$ the Onsager-Machlup function for $\mu^\varepsilon$. Then the limit 
\begin{equation*}
    \lim_{\varepsilon\to 0^+} \varepsilon^2 \operatorname{OM}_\varepsilon(z)=\operatorname{FW}(z). 
\end{equation*}
holds, both as pointwise and in the $\Gamma$-limit sense. 
\end{proposition}
\begin{proof}
By Lemma \ref{lemma:tilt-general}, we have that the Freidlin-Wentzell rate function for $\mu^\varepsilon$ exists and is $    \operatorname{FW}(z)=F_0(z)+\frac{1}{2}\|z\|_{\mu_0}^2-\inf_{z\in \mathcal H_{\mu_0}}(F_0(z)+\frac{1}{2}\|z\|_{\mu_0}^2).$ Then we just need to check the pointwise and $\Gamma$ convergence. Without loss of generality, we may assume that $\inf_{z\in \mathcal H_{\mu_0}}(F_0(z)+\frac{1}{2}\|z\|_{\mu_0}^2)=0$. This is because the Onsager-Machlup function is only defined up to an additive constant and we may add $-\frac{1}{\varepsilon^2}\inf_{z\in \mathcal H_{\mu_0}}(F_0(z)+\frac{1}{2}\|z\|_{\mu_0}^2)$ to $\operatorname{OM}_\varepsilon$. Proceeding with that, using 
%the chain rule for Onsager-Machlup,
Corollary~\ref{corollary:OM-Gaussian-Equivalent} we have that
$$\varepsilon^2 \operatorname{OM}_\varepsilon(z)=F_0(z)+\varepsilon F_1(z)+...+\varepsilon^n R_n(\varepsilon,z)+\frac{1}{2}\|z\|_{\mu_0}^2,$$
as the Onsager-Machlup function for $\mu_0^\varepsilon$ is $\frac{1}{2\varepsilon^2}\|z\|_{\mu_0}^2$. 
Clearly the pointwise limit of this is $F_0(z)+\frac{1}{2}\|z\|_{\mu_0}^2$. Now we just have to check $\Gamma$ convergence. To this aim, note that $F_0(z)+\frac{1}{2}\|z\|_{\mu_0}^2$ is continuous and $\Gamma$ convergence is stable under continuous pertubations. Therefore we just need to show that
\begin{equation*}
    \Gamma-\lim_{\varepsilon\to 0} \varepsilon F_1(z)+...+\varepsilon^n R_n(\varepsilon,z)=0.
\end{equation*}
Note that the $F_i$ are locally bounded and thus on the neighborhood $N_z$ of the point $z$, we have that 
\begin{equation*}
    \lim_{\varepsilon\to 0} \inf_{x\in N_z} \varepsilon F_1(x)+...+\varepsilon^n R_n(\varepsilon,x)=0.
\end{equation*}
Since $N_z$ is an arbitrary neighborhood of $z$, it follows by definition that $\operatorname{OM}_\varepsilon$ $\Gamma$-converges to $FW$.
\end{proof}

\begin{proof}[Proof of Theorem~\ref{theorem:main}]
Following Proposition~\ref{prop:FW-first}, by $\Gamma$ convergence, we have that every cluster point of the minimizers of $\varepsilon^2 \operatorname{OM}_{\mu^\varepsilon}$ is a minimizer of $\operatorname{FW}$ (see \cite{Gamma-book}, section 1.5). If additionally we know that the functions $\varepsilon^2 \operatorname{OM}_{\mu^\varepsilon}$ are equicoercive (which is the case with SDEs with $C^1$ drift), by the fundamental theorem of $\Gamma$ convergence (see \cite{Gamma-book} section 1.5), then we have the full version of Theorem \ref{theorem:main}.
\end{proof}

As a consequence of Theorem \ref{theorem:main}, we can specialize to the case of the generalized Girsanov theorem given in e.g. \cite{NuaBook}, Theorem 4.1.2.

\begin{proposition}\label{prop:generalized-Girsanov-corr}
Let $(\mathcal B, \mu_0)$ be a Gaussian Banach space with Cameron-Martin space $\mathcal H_{\mu_0}$ and Cameron-Martin norm $\|\cdot\|_{\mu_0}$. Define $\mu_0^\varepsilon$ as above and consider white noise process $\{W(h):h\in \mathcal H_{\mu_0}\}$ associated to $\mu_0$. Suppose that $H:\mathcal B\to \mathcal H_{\mu_0}$ is a continuous function so that $W(H)$ is defined and suppose that for all $\varepsilon>0$ we have
\begin{equation*}
    E_{\mu_0^\varepsilon}\left[\exp\left(\frac{1}{\varepsilon^2}\left(W(H)-\frac{1}{2}\|H\|_{\mu_0}^2\right)\right)\right]=1.
\end{equation*}
Define the collection of measures $$\mu^\varepsilon=\exp\left(\frac{1}{\varepsilon^2}\left(W(H)-\frac{1}{2}\|H\|_{\mu_0}^2\right)\right)\mu_0^\varepsilon.$$
Then the Friedlin-Wentzell rate function for $\mu^\varepsilon$ exists and is equal to
\begin{equation*}
    \operatorname{FW}(z)=
    \begin{cases}
    \frac{1}{2}\|z-H(z)\|_{\mu_0}^2&\text{ if } z\in \mathcal H_{\mu_0}\\
    \infty &\text{ else}.
    \end{cases}
\end{equation*}
Furthermore we have that $\lim_{\varepsilon\to 0} \varepsilon^2\operatorname{OM}_\varepsilon(z)=\operatorname{FW}(z)$ both pointwise and in sense of $\Gamma$ convergence. 
\end{proposition}
\begin{proof}
Let $i:\mathcal H_{\mu_0}\to  L^2([0,T],\mathbb R)$ be an isomorphic isometry of separable Hilbert spaces. Denote by $z_t^\ast=i^{-1} ( \chi_{[0,t]})$. Then $W(z_t^\ast)$ is a standard Brownian motion. Furthermore, one can verify that for all $h\in \mathcal H_{\mu_0}$ we have 
$$W(\omega, h)=\int_0^T (ih)(t) dW(\omega, z_t^\ast).$$
Therefore we have that
$$W(\omega, H(\omega))=\int_0^T (iH(\omega))(t) dW(\omega, z_t^\ast).$$
Under the measure $\mu_0^\varepsilon$, we can change to Stratonovich integration to get that
$$\int_0^T (iH(\omega))(t) dW(\omega, z_t^\ast)=\int_0^T (iH(\omega))(t) \circ dW(\omega, z_t^\ast)-\frac{\varepsilon^2}{2}[(iH)(\omega),W(\omega,z_t^\ast)](T).$$
The Stratonovich integral is a continuous function of $\omega$ so long as $H$ is, and so is $\frac{1}{2}\|H\|_{\mu_0}^2$, so therefore Proposition \ref{prop:FW-first} applies and 
\begin{equation*}
    \operatorname{FW}(z)=
    \begin{cases}-\int_0^T (iH(z))(t) \circ dW(z, z_t^\ast)+\frac{1}{2}\|H(z)\|_{\mu_0}^2+\frac{1}{2}\|z\|_{\mu_0}^2&\text{ if }z\in \mathcal H_{\mu_0}\\
    \infty&\text{ else}.
    \end{cases}
\end{equation*}
One may note that for $z\in \mathcal H_{\mu_0}$ we have
\begin{align*}
    \int_0^T (iH(z))(t) \circ dW(z, z_t^\ast)&=-\langle (iH)(z),(iz)\rangle_{L^2}\\
    &=-\langle H(z), z\rangle_{\mu_0}.
\end{align*}
Therefore we arrive at 
\begin{equation*}
    \operatorname{FW}(z)=
    \begin{cases}
    \frac{1}{2}\|z-H(z)\|_{\mu_0}^2&\text{ if } z\in \mathcal H_{\mu_0}\\
    \infty &\text{ else}.
    \end{cases}
\end{equation*}
Finally, note that 
\begin{equation*}
    \operatorname{OM}_\varepsilon(z)=
    \begin{cases}
    \frac{1}{2\varepsilon^2}\|z-H(z)\|_{\mu_0}^2+\frac{1}{2}[(iH)(z), W(z,z_t^\ast)](T)&\text{ if } z\in \mathcal H_{\mu_0}\\
    \infty &\text{ else}.
    \end{cases}
\end{equation*}
\end{proof}
%\sloppy
\begin{remark}
    Note that the remainder term $$\varepsilon^2   \operatorname{OM}_\varepsilon(z)-\operatorname{FW}(z)=\frac{\varepsilon^2}{2}[(iH)(z), W(z,z_t^\ast)](T)$$ could be seen as a test of whether $\mu^\varepsilon$ is Gaussian or not. For example, if $iH(z)=b(z(t))$ for some sufficiently regular $b:\mathbb R\to \mathbb R$, as is the case with SDEs, then $$\frac{\varepsilon^2}{2}[(iH)(z), W(z,z_t^\ast)](T)=\frac{\varepsilon^2}{2}\int_0^T b'(z(t))dt.$$ Which is constant as a function of $z$ if and only if $b'$ is constant. That is, if $W(H)-\frac{1}{2}\|H\|_{\mu_0}^2$ is a quadratic function. 

    This is not always the case, as we can show. Letting $\mu_0$ be the law of a Brownian motion on the space of continuous functions, we can consider the density defined by 
    \begin{equation*}
        \Psi(B)=\exp\left(\frac{1}{\varepsilon^2}\left(\int_0^T \phi(t) dB(t)-\frac{1}{2}\int_0^T \phi^2(t) dt\right)\right),
    \end{equation*}
    where $\phi$ is the adapted process defined by
    \begin{equation*}
        \phi(t)=
        \begin{cases}
            B(t)&\text{ if } t\in [0,T/2]\\
            -B(T-t)&\text{ if } t\in [T/2,T].
        \end{cases}
    \end{equation*}
    Then converting from It\^o to Stratonovich yields zero quadratic covariation with $B(t)$, and thus the functions are equal but the measures $\mu^\varepsilon:=\Psi(B)\mu_0^\varepsilon$ are not Gaussian.
\end{remark}
\subsection{Examples}
\subsubsection{Stochastic Differential Equations}
The principal examples of measures satisfying Theorem \ref{theorem:main} are the laws of solutions to stochastic differential equations. There is the classical theory of SDEs driven by Brownian motion that motivates our proof, but our result also extends to stochastic differential equations driven by more general Gaussian processes (see e.g. \cite{Budhiraja-LDP-FBM}), path dependent SDEs (see e.g. \cite{Ma-path-dependent}), the solution to stochastic PDEs driven by Gaussian fields (see e.g. \cite{LDP-SPDE}), among other SDEs. 

We provide one example here for a path-dependent SDE to demonstrate the potential utility of our result. In order to compute explicitly, we restrict to a linear path-dependent SDE which is Gaussian so our full machinary isn't necessary. Nonetheless, modulo solving nonlinear ODEs as in \cite{Ma-path-dependent}, this general approach should in principle be able to be applied to more general path-dependent SDEs. Consider the stochastic process
\begin{equation*}
    X^\varepsilon(t)=\int_0^t a(s)\varepsilon B(s) ds+\varepsilon B(t),
\end{equation*}
where $a$ is continuous. As $X^\varepsilon= \varepsilon X^1$ is Gaussian we know that its Onsager-Machlup and Freidlin-Wentzell functions satisfy $\operatorname{FW}(\cdot)=\varepsilon^2 \operatorname{OM}_\varepsilon(\cdot)=\frac{1}{2}\|\cdot\|_\mu^2$, where $\|\cdot\|_\mu$ is the Cameron-Martin norm associated to $X^1$, but we can still apply Girsanov and our Theorem \ref{theorem:main}.

An application of Girsanov shows that the law of $X^\varepsilon(t)$, $\mu^\varepsilon$ with respect to the law of $\varepsilon B(t)$, $\mu_0^\varepsilon$ is given by

\begin{equation}\label{eq:example-1-girsanov}
\begin{split}
    \frac{d\mu^\varepsilon}{d\mu_0^\varepsilon}=&\exp\left(\frac{1}{\varepsilon^2}\left(\int_0^T \int_0^t e^{-(A(t)-A(s))}dB(s)dB(t) \right.\right. \\ &\qquad\qquad\left.\left. -\frac{1}{2}\int_0^T \left(\int_0^t e^{-(A(t)-A(s))}dB(s)\right)^2 ds\right)\right),
    \end{split}
\end{equation}
where $A'=a$. Converting to Stratonovich integration gives that $\mu_0^\varepsilon$-a.s. we have 
\begin{equation*}
\begin{split}
    \frac{d\mu^\varepsilon}{d\mu_0^\varepsilon}&=\exp\left(\frac{1}{\varepsilon^2}\left(\int_0^T \int_0^t e^{-(A(t)-A(s))}\circ dB(s)\circ dB(t)\right.\right.\\&\qquad\qquad\left. \left.-\frac{1}{2}\int_0^T \left(\int_0^t e^{-(A(t)-A(s))}\circ dB(s)\right)^2 ds\right)\right),
\end{split}
\end{equation*}
so the Onsager-Machlup function is
\begin{equation*}
    \operatorname{OM}_\varepsilon(z)=\frac{1}{2\varepsilon^2}\int_0^T \left(z'(t)-\int_0^t e^{-(A(t)-A(s))} z'(s) ds \right)^2dt=\frac{1}{2\varepsilon^2}\|z\|_\mu^2.
\end{equation*}
Our Theorem \ref{theorem:main} shows that $\varepsilon^2\operatorname{OM}_\varepsilon(\cdot)\to \operatorname{FW}(\cdot)$ both pointwise and as a $\Gamma$-limit (which for emphasis we repeat is known from the fact that $X^\varepsilon$ is Gaussian). For a more general path-dependent SDE
$$dX(t)=b(X(t),B(t),t)dt+dB(t)$$
the density in equation \eqref{eq:example-1-girsanov} is more complicated and in general cannot be written down explicitly. However Girsanov ensures that it exists and a similar procedure would yield the same result.

\subsubsection{System of Random Algebraic Equations}
We conclude with an example demonstrating that the utility of our result extends beyond the situation of SDEs. Let $a_n$ be a sequence of real numbers that is square summable. Let $\xi_n$ be a sequence of i.i.d. standard normal random variables. Then by \cite{hairer2009introduction}, exercise 3.5 we have that the law of $\mathbf g :=(a_1 \xi_1,x_2\xi_2,...)$, $\mu_0$, is a Gaussian measure on the Banach (Hilbert) space $\mathcal B$ of square summable sequences. The Cameron-Martin space of $\mu_0$, $\mathcal H_{\mu_0}$, is the collection of all sequences $\mathbf z= \{a_n^2 \phi_n\}$ for some square summable sequence $\mathbf \phi = \{\phi_n\}$. For each $\mathbf z \in \mathcal H_{\mu_0}$, we have that $W(\mathbf z)=\langle\mathbf \phi, \mathbf g\rangle =\sum_{n=1}^\infty \phi_n a_n \xi_n$ where $W(\mathbf z)$ is understood as in Proposition \ref{prop:generalized-Girsanov-corr}. The Onsager-Machlup for $\mu_0$ by Theorem \ref{corollary:OM-Gaussian-Equivalent} is $\operatorname{OM}_{\mu_0}(\mathbf z)=\frac{1}2\sum_{n=1}^\infty \phi_n^2 a_n^2=\frac{1}2\|\mathbf z\|_{\mu_0}^2$. Let $f_n$ be measurable functions so that $f_n(\xi_n)$ satisfy for all $\varepsilon>0$
\begin{equation*}
    E_{\mu_0}\left[e^{\frac{1}{2\varepsilon^2} \sum f_n^2(\varepsilon \xi_n)a_n^2}\right]<\infty.
\end{equation*}
Denote the law of $\mathbf g^\varepsilon=(a_1 \varepsilon \xi_1, a_2 \varepsilon \xi_2,...)$ by $\mu_0^\varepsilon$ and define the measures
\begin{equation*}
    \mu^\varepsilon=\exp\left(\frac{1}{\varepsilon^2} \left(\sum_{n=1}^\infty f_n(\xi_n) a_n \xi_n-\frac{1}2\sum_{n=1}^\infty f_n^2(\xi_n)a_n^2\right)\right)\mu_0^\varepsilon.
\end{equation*}
Suppose that there exist solutions $x_n$ to the random algebraic equations $x_n^\varepsilon=f_n(x_n^\varepsilon)+\varepsilon a_n \xi_n$. Then $\mu^\varepsilon$ is the law of the random sequence $\mathbf x^\varepsilon=\{x_n^\varepsilon\}$. We have that $\mathbf x^\varepsilon$ satisfies a LDP with rate function $\operatorname{FW}_\mu(\mathbf z)=\frac{1}2\sum_{n=1}^\infty (\phi_n-f(\phi_n))^2 a_n^2=\varepsilon^2 \operatorname{OM}_{\mu^\varepsilon}(\mathbf z)$.

\bibliographystyle{plain}
\bibliography{bibliography}

\end{document}